\documentclass[11pt]{amsart}

\usepackage{xspace}

\usepackage{amsmath}
\usepackage{amsthm}
\usepackage{amssymb}

 \usepackage{xcolor}

 \usepackage{url}
 \usepackage{graphicx} 
 \usepackage[all]{xy} 

\allowdisplaybreaks 

\newtheorem{theorem}{Theorem}[section]
\newtheorem{proposition}[theorem]{Proposition}
\newtheorem{lemma}[theorem]{Lemma}
\newtheorem{corollary}[theorem]{Corollary}

\theoremstyle{definition}
\newtheorem{definition}[theorem]{Definition}
\newtheorem{example}[theorem]{Example}

\theoremstyle{remark}
\newtheorem{remark}{Remark}

\numberwithin{figure}{section}

\newcommand{\subom}{_{\vert\Omega}}

\newcommand{\cat}{\mathop{\mathrm{cat}}}

\newcommand{\TC}{\mathop{\mathrm{TC}}}
\newcommand{\sd}{\mathop{\mathrm{sd}}}
\newcommand{\inv}{^{-1}}

\newcommand{\geo}[1]{\vert #1 \vert}
\newcommand{\catpro}[1]{#1\, \Pi\, #1}

\newcommand{\scat}{\mathop{\mathrm{scat}}}
\newcommand{\secat}{\mathop{\mathrm{secat}}}
\newcommand{\hsecat}{\mathop{\mathrm{hsecat}}}

\title[Discrete Topological complexity]{Discrete Topological complexity}

\author{D.~Fern\'andez-Ternero \and E.~Mac\'{\i}as-Virg\'os \and E.~Minuz \and J.A.~Vilches}

\date{\today}

\begin{document}


\begin{abstract}We introduce a notion of discrete topological complexity in the setting of simplicial complexes, using only the combinatorial structure of the complex by means of the concept of contiguous simplicial maps. We study the links of this new invariant with those of simplicial and topological LS-category.
\end{abstract}



\maketitle

\setcounter{tocdepth}{1}
\tableofcontents

\section{Introduction}

Topological complexity, introduced by Farber \cite{FARBER1}, is a topological invariant defined to solve
problems in robotics such as motion planning. For this purpose one needs an algorithm that, for each  pair of points of the so-called   configuration space of a mechanical or physical device, computes a path connecting them,  in a continuous way. The  key idea was to interpret
that algorithm in terms of a section of the so-called path-fibration, which is a well-known map in algebraic topology.

The aim of the present paper is to establish a discrete version of this approach. This  is interesting because many motion planning methods transform a continuous problem into a discrete one.  Finite simplicial complexes are the proper setting to develop a discrete version of topology.
The main technical point is  to avoid the  construction of a path-space $PK$ associated to the simplicial complex $K$. To do so, we use a different but equivalent characterization of topological complexity, as explained  in Section \ref{UNO}.

In Section \ref{INVARIANT} we prove that the new invariant $\TC(K)$ only depends on the strong homotopy type of $K$, as defined by Barmak and Minian \cite{BM}. In Section \ref{CATEGORY} we compare this new invariant with  the simplicial LS-category of $K$, defined by us in two previous papers \cite{FMV1,FMMV2}, thus giving a simplicial version of Farber's well known results \cite{FARBER1}. Finally, in Section \ref{GEOM}, $\TC(K)$ is compared with the topological complexity $\TC(\vert K \vert)$ of the geometric realization of the complex $K$.

\medskip

\paragraph{\em Acknowledgements}
We thank Nick Scoville for useful conversations, and Jes\'us Gonz\'alez for pointing out us the reference \cite{JESUSGONZ}. Corollary \ref{SAMEHOM} was pointed out to the second author by John Oprea and inspired our definition of the discrete
topological complexity.

The first and the fourth authors were partially supported by MINECO Spain Research Project MTM2015-65397-P and Junta de Andaluc\'{\i}a Research Groups FQM-326 and FQM-189. The second  author was partially supported by MINECO Spain Research Project MTM2016-78647-P and FEDER and by Xunta de Galicia GPC2015/006. The third author was partially supported by
DFF-Research Project Grants from the Danish Council for Independent Research.

\section{Preliminaries} \label{UNO}

\subsection{Topological complexity}\label{TOPBACK}
We include here some motivational remarks.

Farber's  topological complexity \cite{FARBER1,FARBER2}  is  a particular case of the  \v{S}varc genus or sectional category of a map \cite{CLOT,SVARC}.
\begin{definition}
The {\em \v{S}varc genus} $\secat(f)$ of a map $f\colon X \to Y$  is the minimum integer number $n\geq 0$ such that the codomain $Y$ can be covered by open sets $V_0,\dots,V_n$ with the property that over each $V_j$ there exists a local section $s_j$ of $f$ (that is, a continuous map $s_j\colon V_j \to X$ such that $f\circ s_j=\iota_j$, where $\iota_j\colon V_j\subset Y$ is the inclusion).
\end{definition}

\begin{definition}The {\em topological complexity} of a topological space $X$ is $\TC(X)=\secat(\pi)$, where $\pi\colon PX \to X\times X$ is the  so-called path fibration, that is, the map sending an arbitrary path $\gamma\colon [0,1] \to X$ into the pair $(\gamma(0),\gamma(1))$ formed by the initial and the final points of the path.
\end{definition}

\begin{remark}
It is common in algebraic topology  to consider a normalized version of concepts such as \v{S}varc genus, topological complexity and  LS-category $\cat X$ is often used, as in \cite{CLOT}, in such a way that contractible spaces have category zero. This is the convention we followed in our papers \cite{FMV1,FMMV2} and we will maintain  it here. However, sometimes a non-normalized definition (which is equivalent to $\cat X+1$) can be used in some papers, as Farber did in \cite{FARBER1}.
\end{remark}

An  important result is that for some topological spaces (including  the geometric realization of  any finite simplicial complex) the topological complexity can be computed by taking {\em closed} subspaces  instead of open subspaces. This is discussed in \cite[Chapter 4]{FARBER2}.

Now we proceed to modify the definition of sectional category .
\begin{definition}
The {\em homotopic \v{S}varc genus} of the map $f\colon X \to Y$, denoted by $\hsecat(f)$,  is the minimum integer number $n\geq 0$ such that there exists an open covering of the codomain $Y=V_0\cup\dots\cup V_n$, with the property that for each $V_j$ there exists a local {\em homotopic} section $s_j$, that is, a continuous map $s_j\colon V_j \to X$ such that  there is a homotopy $f\circ s_j\simeq \iota_j$, where $\iota_j\colon V_j\subset Y$ is the inclusion.
\end{definition}
Clearly $\hsecat(f)\leq \secat(f)$. For a particular class of maps both invariants coincide.

\begin{proposition}If $\pi\colon X\to Y$ is a fibration (that is, a map with the homotopy lifting property) then $\hsecat(\pi)=\secat(\pi)$. In particular this is true for the path fibration $\pi \colon PX \to X\times X$.
\end{proposition}

Now, it is well known that any map factors, up to homotopy equivalence, through a fibration. We will apply it to the particular case of the diagonal map $\Delta_X \colon X \to X\times X$.

\begin{proposition} There is a homotopy equivalence $X\simeq PX$ such that the   diagram in Figure \ref{DIAGPROJ} commutes up to homotopy (the maps are $c(x)=x$, the constant path, and $\alpha(\gamma)=\gamma(0)$, the initial point).
\begin{figure}[h]
        $\xymatrix{
        X\ar@<+1ex>[r]^{\ c\ }\ar@<-2ex>[rd]_{\Delta_X}&\ PX\ \ar[d]^{\pi}\ar@<+1ex>[l]^{\ \alpha\
        } \\
        &\ \ X\times X\
        }$
        \caption{} \label{DIAGPROJ}
        \end{figure}
\end{proposition}

\begin{corollary}\label{SAMEHOM}The maps $\pi$ and $\Delta_X$ have the same homotopic \v{S}varc genus, and both coincide with the topological complexity of $X$,
$$\hsecat(\Delta_X)=\hsecat(\pi)=\secat(\pi)=\TC(X).$$
\end{corollary}

\begin{proposition}\label{EQUIV}Let $U\subset X\times X$ be an open subset. The following conditions are equivalent.
\begin{enumerate}
\item
There is a section $s_U\colon U \to PX$ of the path fibration $\pi$;
\item
the restrictions to $U$ of the projections $p_1,p_2\colon X\times X \to X$ are homotopic maps\label{DOS};
\item
either ${p_1}_{\mid U}$ or ${p_2}_{\mid U}$ is a section (up to homotopy) of the diagonal map $\Delta_X\colon X\to X\times X$.
\end{enumerate}
\end{proposition}


\subsection{Simplicial complexes}\
We refer the reader to Kozlov's book \cite{KOZLOV} for a modern survey of simplicial complexes and to Spanier's book \cite{SPANIER}, as well as to our paper \cite{FMV1}, for the classical notions of simplicial maps, simplicial approximation and contiguity.

Let $K$ be a finite abstract simplicial complex.
Let $K^2=\catpro{K}$ be the categorical product as defined in \cite[Definition 4.25]{KOZLOV}.  The set of vertices $V(K^2)$ is $V(K)\times V(K)$, and the simplices of $K^2$ are defined by the rule $\sigma\in K^2$ if and only if $\pi_1(\sigma)$ and $\pi_2(\sigma)$ belong to $K$,
where $\pi_1,\pi_2$ are the projections from $K^2$ into $K$.

Let $\varphi\colon K \to L$ be a simplicial map, and define $\varphi^2=\varphi\,\Pi\,\varphi\colon K^2 \to L^2$ by
$$\varphi^2(v,w)=(\varphi(v),\varphi(w)).$$ A very important property for our purposes is:

\begin{proposition}\label{PRODSIM}
If $\varphi,\psi \colon K \to L$ are simplicial maps in the same contiguity class (denoted by $\varphi\sim\psi$), then  $\varphi^2\sim \psi^2$.
\end{proposition}
\begin{proof}Being in the same contiguity class, $\varphi \sim \psi$, means that there is a sequence of simplicial maps $h_i\colon K \to L$, $i=1,\dots,m$,  such that $h_0=\varphi$, $h_m=\psi$, and the maps $h_i$ and $h_{i+1}$ are contiguous (denoted $h_i\sim_c h_{i+1}$), so we can assume without loss of generality that $\varphi \sim_c \psi$.  By definition it means that for each simplex $\sigma\in K$ the union of vertices $\varphi(\sigma)\cup\psi(\sigma)$ is a simplex of $L$.

Let $\sigma=\{(v_1,w_1),\dots, (v_n,w_n)\}$ be a simplex in $K^2$. By definition, that means that $\pi_1(\sigma)=\{v_1,\dots,v_n\}$ and $\pi_2(\sigma)=\{w_1,\dots,w_n\}$ are simplices of $K$. Then
$$\varphi(\pi_1(\sigma))\cup \psi (\pi_1(\sigma))=\{\varphi(v_1),\dots,\varphi(v_n),\psi(v_1),\dots,\psi(v_n)\}$$ belongs to $L$. Analogously $\varphi(\pi_2(\sigma))\cup \psi (\pi_2(\sigma))\in L$. This is enough to prove that $\varphi^2(\sigma)\cup \psi^2(\sigma)\in L^2$.
\end{proof}

\begin{remark}There is another notion of  simplicial product, the so-called {\em direct product} $K\times K$ where it is necessary to fix an order on $V(K)$. The difference with $\catpro{K}$ is that the geometric realization $\vert K \times K\vert$ is homeomorphic to $\vert K \vert \times \vert K \vert$, while $\vert \catpro{K} \vert$ has only the homotopy type of the latter. However, Proposition \ref{PRODSIM}  would only be true for the direct product  if the maps $\varphi,\psi$ preserve the order.
\end{remark}

\begin{remark}
Recently, Gonz\'alez \cite{JESUSGONZ} introduced a combinatorial version $SC(K)$ of the topological complexity which is based on a simplicial analog of part \eqref{DOS} of Proposition \ref{EQUIV}. However, his notion is based on the direct product $K\times K$ and it seems  not easy to compare it with our notion of simplicial complexity.
\end{remark}

\section{Discrete topological complexity}\label{INVARIANT}
In Section \ref{TOPBACK} we  have explained the reason of the following definitions, which avoid the need of a simplicial version $PK$ of the path space.

\subsection{Farber subcomplexes}
Let $\Omega\subset K^2$ be a simplicial subcomplex of the product $K^2=\catpro{K}$ and let $\iota_\Omega\colon \Omega \subset K^2$ be the inclusion map.

Let $\Delta \colon K \to K^2$ be the diagonal map $\Delta(v)=(v,v)$.

\begin{definition}We say that $\Omega\subset K^2$ is a {\em Farber subcomplex} if there exists a simplicial map $\sigma\colon \Omega\subset K^2 \to K$ such that $\Delta\circ \sigma \sim \iota_\Omega$.
\end{definition}
The map $\sigma$ will be called a {\em local homotopic section} of the diagonal, where ``homotopic''  must be understood in the sense of  belonging to the same contiguity class.

\begin{definition} The {\em discrete  topological complexity} $\TC(K)$ of the simplicial complex $K$ is the least integer $n\geq 0$ such that $K^2$ can be covered by $n+1$ Farber subcomplexes.

In other words, $\TC(K)\leq n$ if and only if $K^2=\Omega_0\cup\cdots\cup \Omega_n$, and there exist simplicial maps $\sigma_j\colon \Omega_j \to K$ such that $\Delta\circ \sigma_j\sim \iota_j$, where $\iota_j\colon \Omega_j\subset K^2$, for $j=0,\dots,n$, are inclusions.
\end{definition}

Sometimes we shall call $\TC(K)$ the {\em simplicial complexity} of $K$ (not to be confused with the notion $SC(K)$ defined by Gonz\'alez in \cite{JESUSGONZ}). Notice that $\TC(K)$ is defined in purely combinatorial terms, involving neither the geometric realization $\geo{K}$ of the complex, nor the notion of topological homotopy, nor that of simplicial approximation.

\subsection{Motion planning}

Farber's  complexity  is a topological invariant introduced to solve
problems in robotics such as motion planning \cite{FARBER2}.
In this section we explain how our notion of  discrete topological complexity is related to the motion planning problem on a simplicial complex.

Let $\Omega\subset K^2$ be a Farber simplicial subcomplex and let $\sigma\colon \Omega  \to K$
be the associated  section (up to contiguity) of the diagonal, that is, such that $\Delta\circ \sigma \sim \iota_\Omega$. Then for each pair of points $x,y\in K$ such that $(x,y)\in \Omega$, the point $\sigma(x,y)$ is an {\em intermediate point} between $x$ and $y$ in the following sense:
consider the sequence of contiguous maps $h_0\sim_c\cdots\sim_c h_j\sim_c\cdots\sim_c h_m$ connecting $\Delta\circ \sigma$ and $\iota_\Omega$. Denote $h_j(x,y)=(x_j,y_j)$. Then $x_m=x$, $y_m=y$ and $x_0=\sigma(x,y)=y_0$. That means that we have a sequence of points
\begin{equation}\label{PATH}
x=x_m,\dots, x_0=\sigma(x,y)=y_0,\dots,y_m=y.
\end{equation}
Moreover, contiguity implies that two consecutive points in the above sequence belong to the same simplex: in fact, since $h_j\sim_c h_{j+1}$, the points $h_j(x,y)=(x_j,y_j)$ and $h_{j+1}(x,y)=(x_{j+1},y_{j+1})$ generate a simplex of $K^2$ (that is, they are either equal or the vertices of an edge). By definition of the product $K^2$, this   means that the points $x_j$ and $x_{j+1}$ (resp. $y_j$ and $y_{j+1}$) generate a simplex of $K$.
Hence the sequence (\ref{PATH}) gives an edge-path on $K$ connecting the points $x$ and $y$.


\subsection{Invariance}
       Recall from \cite{BM} that two simplicial complexes $K,L$ have the same ``strong homotopy type'', $K\sim L$,  if there is  a sequence of elementary strong collapses and expansions connecting them. This is equivalent  to the existence of simplicial maps $\varphi\colon K \to L$ and $\psi\colon L \to K$ such that $\varphi\circ\psi\sim 1_L$ and $\psi\circ\varphi\sim1_K$ (we recall that $\sim$ means ``being in the same contiguity class'').

        \begin{theorem} The discrete topological complexity is an invariant of the strong homotopy type. That is, $K\sim L$ implies  $\TC(K)=\TC(L)$.
        \end{theorem}
        \begin{proof}
        From Prop. \ref{PRODSIM} we have
        $$\varphi^2\circ\psi^2=(\varphi\circ \psi)^2\sim (1_L)^2=1_{L^2}$$ and analogously  $\psi^2\circ\varphi^2\sim 1_{K^2}$, so we have $K^2\sim L^2$. Moreover the  diagram
        in Figure \ref{DIAG}
        \begin{figure}[h]
        $\xymatrix{
        &K\ar@<+1ex>[rr]^{\ \varphi\ }\ar[dd]_{\Delta_K}&&\ L\ \ar[dd]^{\Delta_L}\ar@<+1ex>[ll]^{\ \psi\
        }& \\
        &&&&\\
        \Omega\ \ar@/ ^/[ruu]^{\sigma}\ar@{^{(}->}[r]&K^2\ar@<+1ex>[rr]^{\
        {\varphi^2}}&&\ L^2\ \ar@<+1ex>[ll]^{{\psi^2}\ }&\ar@{_{(}->}[l]\
        \ar@{.>}@/ _/[luu]_{\lambda}\Lambda
        }$
        \caption{} \label{DIAG}
        \end{figure}
        verifies $\Delta_L\circ \varphi=\varphi^2\circ \Delta_K$ and $\Delta_k\circ\psi=\psi^2\circ\Delta_L$.

        Now let $\Omega\subset K^2$ be a Farber subcomplex of $K^2$, that is, there exists a simplicial map $\sigma\colon \Omega\to K$ such that $\Delta_K\circ \sigma\sim\iota_\Omega$.
        Then the inverse  image $\Lambda=(\psi^2)^{-1}(\Omega)\subset L^2$ is a Farber subcomplex of $L^2$, because (see Figure \ref{DIAG}) the map
        $$\lambda=\varphi\circ\sigma\circ{\psi^2}_{\vert \Lambda}\colon \Lambda\subset L^2\to L$$ verifies
        \begin{align*}
        \Delta_L\circ \lambda=&\Delta_L\circ \varphi\circ\sigma\circ\psi^2\circ \iota_\Lambda\\
        =&\varphi^2\circ \Delta_K\circ\sigma\circ\psi^2\circ\iota_\Lambda\sim \varphi^2\circ \iota_\Omega\circ \psi^2\circ\iota_\Lambda\\
        =&(\varphi^2\circ\psi^2)_{\vert  \Lambda} \sim 1_{L^2}\circ\iota_\Lambda\\
        =&\iota_\Lambda.
        \end{align*}
        Let $\TC(K)\leq n$, that is, there exists a covering $K=\Omega_0\cup\cdots\Omega_n$ where $\Omega_j$, $j=0,\dots,n$, are Farber subcomplexes. Then  the corresponding $\Lambda_j=(\psi^2)^{-1}(\Omega_j)$, $j=0,\dots,n$, form a Faber covering of $L^2$, hence $\TC(L)\leq n$. The other inequality is proved in the same way.
        \end{proof}

We have the following characterization of Farber subcomplexes, which is the simplicial version of Proposition \ref{EQUIV}.

\begin{theorem}\label{PROPIEDADES}Let $\Omega\subset K^2$ be a subcomplex of the categorical product. The following conditions are equivalent:
\begin{enumerate}
\item
$\Omega$ is a Farber subcomplex.
\item
the restrictions to $\Omega$ of the projections are in the same contiguity class, that is, $(\pi_1)\subom\sim (\pi_2)\subom$.
\item
Either $(\pi_1)\subom$ or $(\pi_2)\subom$ is a section (up to contiguity) of the diagonal $\Delta\colon K\to K^2$.
\end{enumerate}
\end{theorem}

\begin{proof}\item[$1\Rightarrow 2)$] If $\Omega\subset K^2$ is a Farber subcomplex, then there exists $\sigma\colon \Omega\to K$ such that $\Delta\circ \sigma\sim \iota_\Omega$. But $\Delta\circ \sigma$ is the map $(\sigma,\sigma)$ defined by $\omega\in\Omega\mapsto (\sigma(\omega),\sigma(\omega))$. On the other hand $\iota_\Omega=(\pi_1\circ\iota_\Omega,\pi_2\circ\iota_\Omega)$. Then
$$(\sigma,\sigma)\sim (\pi_1\circ\iota_\Omega,\pi_2\circ\iota_\Omega)$$
which implies, by composing with the projections, that
$$(\pi_1)\subom=\pi_1\circ\iota_\Omega\sim \sigma \sim \pi_2\circ\iota_\Omega=(\pi_2)\subom.$$
\item[$2\Rightarrow 3)$] If $(\pi_1)\subom\sim(\pi_2)\subom$, define $\sigma\colon\Omega\to K$ by $\sigma=(\pi_1)\subom$. Then $\iota_\Omega(x,y)=(x,y)$, for $(x,y)\in\Omega$,  while $(\Delta\circ\sigma)(x,y)=(x,x)$. We  have by hypothesis
$$\iota_\Omega=((\pi_1)\subom,(\pi_2)\subom)\sim ((\pi_1)\subom,(\pi_1)\subom)=\Delta\circ\sigma.$$
\item[$3\Rightarrow 1)$] If $\sigma=(\pi_i)\subom$ verifies $\Delta\circ\sigma\sim \iota_\Omega$, then $\Omega$ is a Farber subcomplex, by definition.
\end{proof}


\section{Relationship with simplicial LS-category}\label{CATEGORY}
One of Farber's main  results for topological complexity   relates it to a well known classical invariant, the Lusternik-Schnirelmann category \cite{CLOT}. In this section we get analogous results for the discrete setting, by using the simplicial LS-category of a simplicial complex  introduced by the authors in \cite{FMV1, FMMV2}.

\subsection{Comparison with the category of $K$}

\begin{definition}Let $K$ be an abstract simplicial complex. A subcomplex $L\subset K$ is {\em categorical} if the inclusion $\iota_L\colon L\subset K$ belongs to the contiguity class of some constant map $L\to K$, that is, $\iota_L\sim \ast$. The (normalized) simplicial {\em LS-category} $\scat K$ of the simplicial complex $K$ is the minimum number $m\geq 0$  such that there are categorical subcomplexes $L_0,\dots,L_m$ which cover $K$, that is, $K=L_0\cup\cdots \cup L_m$.
\end{definition}

\begin{remark}
As explained in \cite{FMV1}, a categorical subcomplex  may not be strongly collapsible in itself, but it must be in the ambient complex.  Equivalently,  it is the inclusion $\iota_L$, and not  the identity $1_L$, which belongs to the contiguity class of a constant map.
\end{remark}

The first inequality  proved by Farber directly compares the topological complexity $\TC(X)$ of a space with the LS-category $\cat X$. We shall prove that this result also holds  in the discrete setting.

\begin{theorem}For any  abstract simplicial complex we have
$$\scat K\leq \TC(K).$$
\end{theorem}

\begin{proof}If $\TC(K)\leq n$, let $K^2=\Omega_0\cup\cdots\cup\Omega_n$ be a covering by Farber subcomplexes. Fix a base point $v_0\in K$ and let $i_0\colon K \to K^2$ be the simplicial map $i_0(w)=(v_0,w)$. Then, let us take the inverse images
$$\Sigma_j =(i_0)^{-1}(\Omega_j)\subset K, \quad j=0,\dots,n.$$
Since $K=\Sigma_0\cup\cdots\cup\Sigma_n$, if we   prove that each $\Sigma_j$ is a categorical subcomplex then we can conclude that $\scat K\leq n$, and the result follows.

Let $\Omega\subset K^2$ be a Farber subcomplex, with a local section $\sigma\colon \Omega \to K$ such that $\Delta_K\circ \sigma\sim \iota_\Omega$, and let $\Sigma=(i_0)^{-1}(\Omega)\subset K$. We shall prove that the inclusion $\iota_\Sigma\colon \Sigma\subset K$ belongs to the contiguity class of the constant map $v_0\colon \Sigma\to K$, so we shall obtain that $\Sigma$ is a categorical subcomplex of $K$.

Since $\Delta_K\circ \sigma \sim \iota_\Omega$, there is a sequence of contiguous maps $\psi_i\colon \Omega \to K^2$, $i=1,\dots,m$, such that
\begin{equation}\label{CHAIN1}
\Delta_K\circ\sigma=\psi_1\sim_c\cdots\sim_c \psi_m=\iota_\Omega.
\end{equation}
Then, by composition,
$$\pi_1\circ\psi_1\circ i_0\circ\iota_\Sigma\sim_c\dots\sim_c\pi_1\circ\psi_m\circ i_0\circ \iota_\Sigma,$$
where, for every $w\in \Sigma$,
$$\pi_1\circ\psi_1\circ i_0\circ \iota_\Sigma(w)=\pi_1\circ\Delta_K\circ\sigma\circ i_0(w)=\sigma(v_0,w),$$
and
$$\pi_1\circ\psi_m\circ i_0\circ \iota_\Sigma(w)=\pi_1\circ\iota_\Omega(v_0,w)=v_0.$$
On the other hand
\begin{equation}\label{CHAIN2}
\pi_2\circ\psi_1\circ i_0\circ\iota_\Sigma\sim_c\dots\sim_c\pi_2\circ\psi_m\circ i_0\circ \iota_\Sigma,
\end{equation}
where, for every $w\in \Sigma$,
$$\pi_2\circ\psi_m\circ i_0\circ\iota_\Sigma(w)=\pi_2\circ \iota_\Omega(v_0,w)=w,$$
and
$$\pi_2\circ\psi_1\circ i_0\circ\iota_\Sigma(w)=\pi_2\circ\Delta_K\circ\sigma\circ i_0(w)=\sigma(v_0,w).$$
From (\ref{CHAIN1}) and (\ref{CHAIN2}) it follows
$$v_0\sim \sigma(v_0,w) \sim w, \quad \forall w\in \Sigma,$$
or equivalently, $v_0\sim \iota_\Sigma$, hence $\Sigma$ is a categorical subcomplex.
\end{proof}


\subsection{Comparison with the category of $K^2$}
The second comparison  result by Farber in \cite{FARBER1}  is between $\TC(X)$ and $\cat (X\times X)$. We shall prove that it is  also true in the discrete setting.

\begin{lemma}\label{CONNECTED}The abstract simplicial complex $K$ is edge-path connected if and only if two arbitrary constant maps $L\to K$ are in the same contiguity class.
\end{lemma}

The following theorem uses the normalized versions of LS-category and  topological complexity.

\begin{theorem}  If  $K$ is an edge-path connected  complex, then
$$\TC(K) \leq \scat (K^2).$$
\end{theorem}
\begin{proof}Let $\scat(K\,\Pi\, K)=n$ and let $K^2=\Omega_0\cup\cdots\Omega_n$ be a categorical covering of $K^2$. If we are able to prove that each $\Omega=\Omega_j$, $j=0,\dots,n$,  is a Farber subcomplex then we will have $\TC(K)\leq n$, thus proving the Theorem.

By definition the inclusion $\iota_\Omega\colon \Omega\subset K^2$ verifies $\iota_\Omega\sim \ast$, where $\ast\colon \Omega \to K^2$ is some constant map $(v_0,w_0)$. Since the complex is path-connected we can choose the point $\ast$ verifying  $w_0=v_0$.

By definition of contiguity class, since $\iota_\Omega\sim \ast$,  there is a sequence of simplicial maps, each one contiguous to the next one,
$$\iota_\Omega=\varphi_1\sim_c\cdots\sim_c \varphi_m=(v_0,v_0),$$ with $\varphi_j\colon \Omega \to K^2$. Let $\pi_1\colon K^2 \to K$ the projection onto the second factor, then each $\pi_1\circ \varphi_j\colon \Omega \to K$ is contiguous to $\pi_1\circ \varphi_{j+1}$. Hence
\begin{equation}\label{PROV1}
\pi_1\circ\iota_\Omega \sim \pi_1\circ \varphi_m=v_0.
\end{equation}

Analogously, let $\pi_2\colon K^2 \to K$ be the projection onto the first factor,  then
\begin{equation}\label{PROV2}
\pi_2\circ\iota_\Omega \sim \pi_2\circ \varphi_m=v_0.
\end{equation}
by means of the sequence $\pi_2\circ\varphi_j$.

Now, we shall verify that the map $\sigma=(\pi_1)_{\vert\Omega}\colon \Omega \to K$ verifies
$\Delta_K\circ \sigma\sim \iota_\Omega$, so we conclude the proof.

Define the maps $\xi_j\colon \Omega \to K^2$, $j=1,\dots,m$, as
$$\xi_j(v,w)=(v,\pi_1\circ \varphi_j(v,w)).$$
These are simplicial maps. Moreover, it is clear that $\xi_1\sim \cdots \sim \xi_m$.

Analogously define $\chi_j\colon \Omega \to K^2$, $j=1,\dots,m$, as
$$\chi_j(v,w)=(v,\pi_2\circ \varphi_j(v,w)).$$
They verify  $\chi_1\sim\cdots\sim\chi_m$.

Then it is immediate to check that:
\begin{enumerate}
\item[i)]
 $\xi_1(v,w)=(v,v)$, that is, $\xi_1=\Delta_K\circ \sigma$;
\item[ii)]
 $\xi_m(v,w)=(v,v_0)$;
\item[iii)]
  $\chi_1(v,w)=(v,w)$, that is $\chi_1=\iota_\Omega$.
\item[iv)]
 $\chi_m(v,w)=(v,v_0)$.
\end{enumerate}
 Then, finally we get:
 $$\Delta_K\circ \sigma=\xi_1\sim \xi_m=\chi_m\sim \chi_1=\iota_\Omega. \qedhere$$
\end{proof}

\begin{corollary}The abstract simplicial complex $K$ is strongly collapsible if and only if  $\TC(K)=0$.
\end{corollary}

\begin{proof}  By definition, $K$ being strongly collapsible is equivalent to $\scat K=0$. Moreover, in \cite{FMMV2} we proved that $\scat K^2+1\leq (\scat K+1)^2$ (in fact, the categorical product of strongly collapsible complexes is strongly collapsible). Then $\TC(K)=0$. The converse is immediate from the inequality $\TC(K)\geq \scat K$.
\end{proof}
\begin{corollary} The diagonal $\Delta\colon K \to K^2$ admits a {\em global} homotopic section (in the sense of contiguity, that is, there exists $\sigma\colon K^2 \to K$ such that $\Delta_K\circ \sigma \sim 1_K$) if and only if the complex $K$ is strongly collapsible.
\end{corollary}

\begin{example}Consider the complex $K=\partial \Delta^2$ given by the simplices
$$K=\{\emptyset, \{a\},\{b\},\{c\}, \{b,c\}, \{a,c\}, \{a,b\}\},$$
whose geometric realization is represented in Figure \ref{TRIANGLE}.
\begin{figure}[h]
\includegraphics[height=20mm]{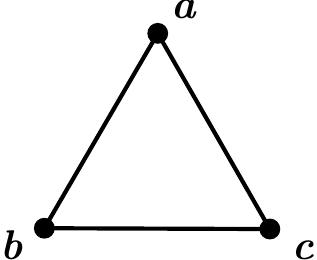}
\caption{} \label{TRIANGLE}
\end{figure}

Since $K$ is not strongly collapsible, but can be covered by two strongly collapsible subcomplexes, it follows that $\scat K=1$. Moreover $\scat K^2+1\leq (\scat K+1)^2=4$ \cite{FMMV2}, hence $1\leq \TC(K)\leq 3$. Then a section $\sigma$ defined in the whole complex $K^2$ is not possible.

It is easy to find three Farber subcomplexes covering $K^2$, and we shall prove now that two are not enough. Then $\TC(K)=2$. In fact, suppose that $K^2=\Omega_1\cup\Omega_2$ is a covering by two subcomplexes. Since $K^2$ has nine maximal simplices (see Figure \ref{BIGPROD}) then one of the subcomplexes, say $\Omega_1$, contains at least five of them.
Now there are nine horizontal edges, so two of the maximal simplices in $\Omega_1$, say $\tau_1$ and $\tau_2$, must have one common horizontal edge. Finally, for each vertex $v_0\in K$, let $i_0\colon K \to K$ be the map $i_0(v)=(v_0,v)$. From Proposition \ref{EQUIV}, that $\Omega_1$ is a Farber subcomplex implies that  the subcomplex
$$(i_0)\inv(\Omega_1) = (\{v_0\}\times K) \cap \Omega_1 \subset K$$
is categorical in $K$, in particular it is not $K$ (because $K$ is not strongly collapsible). That means that $\Omega_1$ can not contain three consecutive vertical edges. Then none of the maximal simplices $P,Q,R$ in Figure \ref{BIGPROD} can be contained in $\Omega_1$. But $\Omega_2$ is also a Farber subcomplex, so it can not contain them as well, because by using the map $i_1(v)=(v,v_0)$ one proves that $\Omega_2$ can not contain three consecutive horizontal edges.

\begin{figure}[h]
\includegraphics[height=45mm]{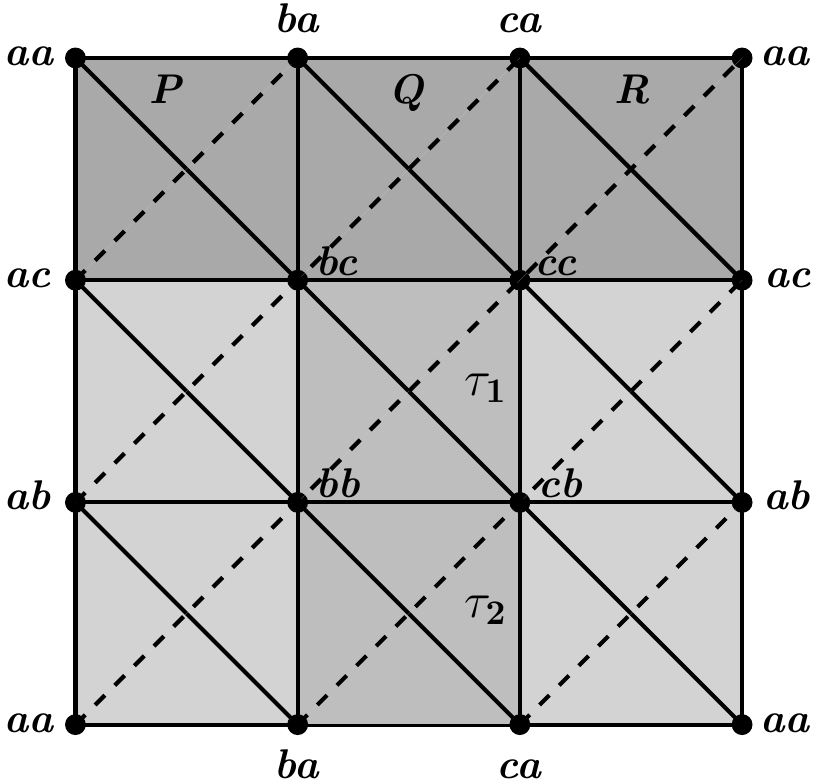}
\caption{} \label{BIGPROD}
\end{figure}
\end{example}


\section{Geometric realization}\label{GEOM}
Let $\geo{K}$ be the geometric realization of the simplicial complex $K$. We can compute the usual topological complexity $\TC(\geo{K})$ of the topological space $\geo{K}$ and to compare it with the discrete (simplicial) complexity  $\TC(K)$ of the simplicial complex $K$.

We need a previous result. It is known that $\geo{K^2}$ is not homeomorphic to the topological product $\geo{K}\times \geo{K}$, but they have the same homotopy type, as proved in Kozlov \cite[Prop.15.23]{KOZLOV}. The proof is based in the so-called ``nerve theorem''. However we need an explicit formula, to guarantee the following lemma.

\begin{lemma}There exists a homotopy equivalence  $u\colon \geo{K}\times\geo{K}\to \geo{K^2}$ satisfying that the projections $p_1,p_2\colon \vert K \vert \times \vert K \vert \to \vert K \vert$ and $\pi_1,\pi_2 \colon \catpro{K} \to K$ verify (up to homotopy) that $\vert \pi_i\vert \circ u =p_i$, for $i=1,2$ (see Figure \ref{KOZLOVTRUE}).\end{lemma}
\begin{proof}There is a homeomorphism $\vert K \times K \vert =\vert K \vert \times \vert K \vert$ which is induced by the projections \cite[p.~538]{HATCHER}. On the other hand, the homotopy equivalence $\vert K\times K \vert \simeq \vert \catpro{K}\vert$
is the geometric realization of the simplicial map $K\times K \to \catpro{K}$ induced by the natural inclusion map $\sigma_1\times \sigma_2 \to \sigma_1\,\Pi\, \sigma_2$ for each pair of simplices $\sigma_1,\sigma_2\in K$ (see \cite[Prop.~15.23]{KOZLOV} and \cite[Prop.4G.2]{HATCHER}).
\begin{figure}[h]
$\xymatrix{
\geo{K}\times\geo{K}\ar@<-2ex>[rd]_{p_i}\ar@<+1ex>[r]^{\ \ u}
&\ar@<+1ex>[l]^{\ \ v}\ \geo{K^2}\ \ar[d]^{\geo{\pi_i}}\\
&\ \geo{K}\ \\
}$
\caption{} \label{KOZLOVTRUE}
\end{figure}
\end{proof}

\begin{theorem}\label{GEOMREALIZ}$\TC(\geo{K})\leq \TC(K)$.
\end{theorem}
\begin{proof}
Let $\TC(K)\leq n$ and  let $K^2=\Omega_0\cup\cdots\cup\Omega_n$ be a Farber covering.

 Let $\Omega$ one of the Farber subcomplexes $\Omega_j$ of the covering of $K^2$, and let $i_\Omega \subset K^2$ be the inclusion. By construction of the geometric realization we have that $\geo{i_\Omega}$ is the inclusion $i_{\geo{\Omega}}\colon \geo{\Omega}\subset \geo{K^2}$. By hypothesis, the maps $\pi_1\circ i_\Omega$ and $\pi_2\circ i_\Omega$ are in the same contiguity class (Proposition \ref{PROPIEDADES}).
By applying the functor $\geo{\cdot}$ of geometric realization, and taking into account that contiguous maps induce homotopic continuous maps (see \cite{SPANIER}), we have
that $\vert \pi_1\vert \circ i_{\vert \Omega\vert}=\vert \pi_1\circ i_\Omega\vert$ is homotopic to $\vert \pi_2\vert \circ i_{\vert \Omega\vert}$.

Consider the closed subspace $F=u^{-1}(\geo{\Omega})\subset \geo{K}\times \geo{K}$. Then the map
$$p_1\circ i_F =\vert \pi_1 \vert\circ u\circ i_F= \vert\pi_1\vert\circ i_{\vert \Omega \vert}$$
is homotopic to $p_2\circ i_F$.
Consider the closed covering
$F_0\cup\cdots\cup F_n$ of $\vert K \vert \times \vert K \vert$.  This implies $\TC(\geo{K})\leq n$.
\end{proof}

\begin{remark}
Notice that the inequality in the latter Theorem is still true for all subdivisions of $K$, because the geometric realizations are homeomorphic, $\geo{\sd K}\cong\geo{K}$. It may happen that $\TC(K)$ differs from $\TC(\sd K)$,  which reflects some particular property of the combinatorial structure.
\end{remark}


\bigskip

\small{

\address{
\noindent {\sc D.~Fern\'andez-Ternero}.
\\Dpto. de Geometr\'{\i}a y Topolog\'{\i}a, Universidad de Sevilla, Spain.\\}
\email{desamfer@us.es}

\medskip

\address{
\noindent {\sc E.~Mac\'ias-Virg\'os}.
\\{Dpto. de Matem\'aticas,} Universidade de San\-tia\-go de Compostela, Spain.\\}
\email{quique.macias@usc.es}

\medskip

\address{
\noindent {\sc E.~Minuz}.
\\Department of Mathematics, Aarhus University, Denmark\\}
\email{minuz@math.au.dk}

\medskip

\address{
\noindent {\sc J.A.~Vilches}.
\\Dpto. de Geometr\'{\i}a y Topolog\'{\i}a, Universidad de Sevilla, Spain.\\}
\email{vilches@us.es}%

}

 \end{document}